\documentclass[11pt]{amsart}

\usepackage{color}

\usepackage{fullpage}
\usepackage{verbatim}
\usepackage{graphics}
\usepackage[pdftex]{graphicx}
\usepackage{amsmath, amssymb, amsfonts, amsthm}
\usepackage{url}
\usepackage{setspace}
\usepackage{algorithmic}
\usepackage{enumitem}
\setlist{nolistsep}

\def \PG[#1,#2]{PG(#1,#2)}
\def \AG[#1,#2]{AG(#1,#2)}

\newtheorem{theorem}{Theorem}

\newtheorem{lemma}[theorem]{Lemma}

\newtheorem{corollary}[theorem]{Corollary}
\newtheorem{proposition}[theorem]{Proposition}

\newcommand{\RR}{\ensuremath{\mathbb R}}

\newcommand{\vrts}{\mathcal S}

\newcommand{\lines}{\mathcal L}

\newcommand{\bis}{\mathcal B}

\newcommand{\en}{\mathcal Q}

\newcommand{\eps}{\varepsilon}
\newcommand{\F}{\mathbb{F}}
\newcommand{\IT}{\mathcal{T}}

\def\minus{\backslash}

\providecommand\phantomsection{} 

\begin{document}

\title{Bisectors and pinned distances.}
\author{Ben Lund \and Giorgis Petridis}
\maketitle

\begin{abstract}
We prove, under suitable conditions, a lower bound on the number of pinned distances determined by small subsets of two-dimensional vector spaces over fields. For finite subsets of the Euclidean plane we prove an upper bound for their bisector energy. 
\end{abstract}

\let\thefootnote\relax\footnotetext{Work on this project by the first author was supported by NSF DMS grant 1344994 and by NSF DMS grant 1802787.
	The second author is supported by the NSF DMS Grant 1723016.}

\section[introduction]{introduction}

A well-known conjecture of Erd\H{o}s \cite{erdos1946sets} is that between the pairs of any $N$ points in $\RR^2$ occur at least $\Omega(N \log^{-1/2} (N))$ different distances.
In 2010, Guth and Katz \cite{guth2015erdos} showed that $\Omega(N \log^{-1}(N))$ distinct distances must occur.

Erd\H{o}s also made a refinement of this conjecture.
This refinement is that among any set $A$ of $N$ points in $\RR^2$, there must occur a point $a \in A$ such that the remaining points of $A$ are at $\Omega(N \log^{-1/2}(N))$ different distances from $a$.
This is often called the pinned distances conjecture.
The best known bound of $\Omega(N^{0.864\ldots})$ was proved by Katz and Tardos \cite{katz2004new}, who improved an earlier proof by Solymosi and T\'oth \cite{solymosi2001distinct}.

These problems have also been considered over other fields, particularly finite fields.
In general, we define the algebraic distance between two points $a,b \in \F^2$ to be
\[\|a-b\| = (a-b) \cdot (a-b) =  (a_1-b_1)^2 + (a_2-b_2)^2. \]
The distance number and pinned distance numbers of $A$ are
\[ \Delta(A) = |\{\|a-b\| : a, b \in A \}| \text{ and } \Delta_{\text{pin}}(A) = \max_{a \in A}|\{\|a-b\| : b \in A \}|.  \]

We study $\Delta_{\text{pin}}(A)$. A simple lower bound of $\Delta_{\text{pin}}(A) = \Omega(N^{1/2})$ comes from the observation that either some single circle contains $N^{1/2}$ points (giving at least $\frac{1}{2}N^{1/2}$ distinct distances from any point in the circle), or there are at least $N^{1/2} - 1$ distinct distances from each point.
This simple bound is tight over fields of positive characteristic, as can be seen by taking our set of points to be the Cartesian product of the prime subfield. In this case $\Delta(A)$ is contained in the prime subfield and hence $\Delta_{\text{pin}}(A) \leq \Delta(A) \leq N^{1/2}$.

There are two ways of proving a non-trivial pinned distance bound over fields of positive characteristic.
You can prove a lower bound that depends on the characteristic of the field, or you can assume that the number of points is small relative to the characteristic of the field.
In this paper, we are primarily interested in point sets that are relatively small.

The first work on the pinned variant of the distinct distance problem over finite fields was done by Bourgain, Katz, and Tao \cite{bourgain2004sum}.
Prior to this paper, the strongest general quantitative bound for small sets was by Stevens and de Zeeuw \cite{stevens2017improved}, who proved $\Delta_{\text{pin}}(A) = \Omega(N^{1/2+1/30})$ over any field of characteristic not equal to $2$.
Iosevich, Koh, Pham, Shen, and Vinh \cite{iosevich2018new} recently gave a stronger bound of $\Delta(A) = \Omega(N^{1/2 + 149/4214})$ for point sets in $\F_p^2$ for prime $p$ with $p \equiv 3 \mod 4 $.
Our main result is quantitatively stronger than that of Iosevich et. al., and applies in the general setting of Stevens and de Zeeuw.

\begin{theorem} \label{Pinned dist ff}
	Let $\F$ be a field with characteristic not equal to $2$.
	There exists an absolute positive constant $c>0$ with the following property. 
	Let $A \subset \F^2$.
	If $\F$ has positive characteristic $p $, then suppose that $|A| \leq p^{8/5}$.
	Then, either
	$$\{\|a - b\| : a,b \in A\} = \{0\},$$
	or there exists $a \in A$ such that
	$$|\{\|a - b\|: b \in A\}| \geq c |A|^{20/37}.$$
\end{theorem}
We remark that our proof of Theorem \ref{Pinned dist ff} remains valid if $\|\cdot\|$ is replaced by an arbitrary quadratic form. This is because there are no restrictions on the field $\F$ (except that it must not have characteristic 2) and so we may take it to be algebraically closed, in which case all quadratic forms are equivaleent.

Comparing the exponent in our bound to the earlier bounds, the Stevens and de Zeeuw exponent is $1/2 + 1/30 \approx 0.533$, the Iosevich et. al. exponent is $1/2 + 149/4214 \approx 0.535$, and our exponent is $1/2 + 3/74 \approx 0.541$.

Proving distance bounds over fields of characteristic $p$ in the case $p \equiv 1 \mod 4$ is usually more complicated than in the case $p \equiv 3 \mod 4$.
This arises from the fact that $-1$ is a square when $p \equiv 1 \mod 4$, so there are nontrivial solutions to $x^2 + y^2 = 0$, and hence the distance between distinct points can be $0$.
For most approaches to the distinct distance problem (such as the one followed in this paper), it is possible to adapt arguments for the $p \equiv 3 \mod 4 $ case to the general case with some effort.
The paper of Iosevich et. al. uses a result of Rudnev and Shkredov \cite{rudnev2018restriction} on the additive energy of a point set on a paraboloid that does not hold in the $p \equiv 1 \mod 4$ case, and so is likely to be harder to generalize.

For the case that $\F = \F_p$ is a finite field of prime order $p$, Hanson, Lund, and Roche-Newton \cite{hanson2016distinct} proved the complementary bound
\[
\Delta_{\text{pin}}(A) \geq c \min\left\{p , \frac{|A|^{3/2}}p \right\}
\]
(their result in fact holds over arbitrary finite fields). 
For $|A| \geq p^{8/5}$ this becomes $\min\{p, |A|^{7/8}\}$. 
Therefore, for prime order finite fields, the current state of affairs may be summarized as: 
\[
\Delta_{\text{pin}}(A) \geq c \min\{p , |A|^{20/37}\}.
\]
Petridis \cite{petridis2017pinned} obtained the better $c \min\{p , |A|^{3/4}\}$ lower bound when $A$ is a Cartesian product.

The key idea in the proof of Theorem \ref{Pinned dist ff} that enables us to improve the bound of Stevens and de Zeeuw is to bound the {\em bisector energy} of the point set.
We say the perpendicular bisector of points $a,b \in \F^2$ with $\|a-b\| \neq 0$ is the line
\[\bis(a,b) = \{x : \|a-x\| = \|b-x\| \}. \]
The bisector energy is the number of quadruples that share a bisector.
In other words, the bisector energy is
\[\en =  |\{(a,b,c,d) \in A^4 : \bis(a,b) = \bis(c,d)  \}. \]

Taking $A$ to be a collection of equally spaced points on a line or the vertices of a regular polygon, shows that $\en$ might be a constant multiple of $|A|^3$. 
We follow a dichotomy argument to prove Theorem~\ref{Pinned dist ff} . 
If $\en \leq |A|^{3-\delta}$ for some absolute $\delta>0$, then the standard argument employed by Stevens and de Zeeuw (which goes back at least to \cite{erdos1975some}) can be improved. 
The key to our result is to show that if, on the other hand, $\en \geq |A|^{3-\delta}$, then there exists a line or a circle incident to many points of $A$. 
Taking any point of $A$ not on such a line or circle as the pin, gives many pinned distances. 
In completing this crucial step we utilize the set up of Lund, Sheffer, and de Zeeuw \cite{lund2015bisector}, first developed to bound bisector energy over $\RR$. 
We also use results of Hanson, Lund, and Roche-Newton \cite{hanson2016distinct}, who bounded the bisector energy for large sets over arbitrary finite fields.

Our second result is progress on a conjecture of Lund, Sheffer, and de Zeeuw \cite{lund2015bisector} on the bisector energy of a point set in $\RR^2$.
Lund, Sheffer, and de Zeeuw made a conjecture that relates $\en$ to the structure of the point set. If $M$ is the maximum number of points of $A$ incident to a line or to a circle, then they conjecture that $\en$ is in the of order $MN^2$. They proved the first non-trivial result in this direction by showing that there exists an absolute positive constant $C>0$  and, for all $\eps>0$, a positive constant $C_\eps >0$ such that  
\[
\en \leq C M N^2 + C_\eps M^{2/5} N^{2+2/5+ \eps}.
\]
This confirms the conjecture for $M \geq N^{2/3 + \delta}$ for any $\delta>0$. We prove an upper bound that is stronger when $M \geq N^{1/4}$, and which confirms the conjecture for $M \geq (\log(N) N)^{1/2}$.

\begin{theorem} \label{Real bis en}
	There exists an absolute positive constant $C>0$ with the following property. Let $A \subset \RR^2$ be a finite set of size $|A|=N$ and $\en$ be its bisector energy. If there are are at most $M$ elements of $A$ incident to any line or any circle, then
	\[
	\en \leq  C (M N^2 + \log^{1/2}(N) \, N^{2 + 1/2}).
	\]
\end{theorem}

The proof of Theorem~\ref{Real bis en} uses the set up, as well as the fundamental geometric observations in \cite{lund2015bisector}. Some aspects of the proof are, however, different. A powerful polynomial partitioning result of Fox et al. \cite{fox2014semi} is not used. As a result of this, the ``main'' term $MN^2$ appears for different reasons than in \cite{lund2015bisector}. 
We depend instead on Beck's theorem \cite{beck1983lattice}.
In this sense Theorem~\ref{Real bis en} has a more elementary proof, though it relies on the full strength of the Guth--Katz theorem for distinct distances \cite{guth2015erdos}. 

Lund, Sheffer, and de Zeeuw \cite{lund2015bisector} also considered the question of how many lines occur as perpendicular bisectors of a set of points in $\RR^2$.
Taking $A$ to be a collection of equally spaced points on a line or the vertices of a regular polygon shows that a set of $N$ points may determine as few as $N$ distinct perpendicular bisectors.
The conjecture of Lund, Sheffer, and de Zeeuw on this question is that, for a set $A$ of $N$ points in $\RR^2$, either $N/2$ points are contained in a single line or circle, or $A$ determines $\Omega(N^2)$ distinct perpendicular bisectors.
Lund, Sheffer, and de Zeeuw showed that a set of $N$ points, no more than $M$ of which are incident to any line or circle, determines $\Omega_\eps (\min (M^{-{2}/{5}} N^{{8}/{5}-\eps}, M^{-1} N^2 ))$ distinct perpendicular bisectors, for any $\eps > 0$, with the implied constant depending on $\eps$.
Lund \cite{lund2016refined} showed that a set of $N$ points in the plane, no more than $N/2$ of which are incident to any line or circle, determines $\Omega_\eps(N^{52/35 - \eps})$ distinct perpendicular bisectors, for any $\eps > 0$, with the implied constant depending on $\eps$.
A slight generalization of Theorem \ref{Real bis en}, combined with the arguments of \cite{lund2016refined}, gives the following improvement to Lund's bound.

\begin{theorem}\label{thm:distinctBisectors}
	There exists a constant $C>0$ with the following property.
	Let $A \subset \RR^2$ be a set of $N$ points.
	Either a single line or circle contains $N/2$ points of $A$, or $C N^{3/2}\log^{1/2}(N)$ distinct lines occur as perpendicular bisectors of pairs of points of $A$.
\end{theorem}

This is strictly better than Lund's bound, and improves the bound of Lund, Sheffer, and de Zeeuw for point sets with at most $N^{1/4}$ points on any line or circle.

\section[Proof of Theorem 1]{Proof of Theorem \ref{Pinned dist ff}}\label{sec:pinnedDistances}

Let us quickly recap the method of Lund, Sheffer, and de Zeeuw \cite{lund2015bisector} to bound bisector energy. For an ordered pair of points $(a,c) \in \F^2 \times \F^2$ with $\|a-c\| \neq 0$, we define
$$S_{ac} = \left \{(b,d): \bis(a,b) = \bis(c,d)\right \} \cup \left \{(a,d) : a \in \bis(c,d) \right \} \cup \left \{(b,d) : d \in \bis(a,b) \right \}.$$
We abuse terminology and call $S_{ac}$ a variety. 
It is important to note that if $(b,d) \in S_{ac}$, then $\|b-d\| = \|a-c\|$. 
For $r \in \F^\times$, we define
\[
\Pi_r = \{(b,d) \in A \times A : \|b-d\| = r \} \text{ and } \vrts_r = \{ S_{ac} : a,c \in A , \|a-c\| = r\},
\]
and denote by
\[
I(\Pi_r, \vrts_r) = |\{ ((b, d) , S_{ac}) : (b,d) \in S_{ac}  \}|
\]
the number of incidences between pairs of points and varieties.
Lund, Sheffer, and de Zeeuw observed that
\[ \en = \sum_{r \in \F^\times} I(\Pi_r, \vrts_r). \]
Our aim is therefore to bound $I(\Pi_r, \vrts_r)$, which is zero unless $r$ lies in the distance set of $A$. We prove a slightly stronger statement.

\begin{proposition} \label{Pts vars inc}
Let $A \subset \F^2$. Suppose at most $M$ points of $A$ are incident to any circle or line. If $P \subseteq \Pi_r$, and $V \subseteq \vrts_r$, then
\[
I(P, V) = O(|P| + M^{2/3} |P|^{2/3} |V|^{2/3} + |P|^{2/3} |V|).
\] 
\end{proposition}

The proof relies on the following lemma, which was proved for $\F = \mathbb{R}$ by Lund, Sheffer, and de Zeeuw.
\begin{lemma} \label{lem:LSdZ}
Let $(a,c) \neq (a',c') \in \F^2 \times \F^2$ with $\|a-c\| = \|a'-c'\| \neq 0$. The intersection $S_{ac} \cap S_{a'c'} \subset \F^2 \times \F^2$ is contained in the Cartesian product of either a pair of concentric circles or a pair of parallel lines. The first circle or line contains both $a$ and $a'$, and the second circle or line contains both $c$ and $c'$. Moreover these circles or lines are uniquely determined by $(a,c)$ and $(a',c')$.
\end{lemma}

The lemma implies that the incidence graph on $P \times V$ contains no $K_{2,M}$ or $K_{M,2}$ (there is clear duality between pairs of points and varieties), which leads to the bounds
\begin{equation}\label{eq:KSTbounds}
I(\Pi_r, \vrts_r) = O(M^{1/2} |P|^{1/2} |V| + |P|) \text{ and } I(\Pi_r, \vrts_r) = O(M^{1/2} |V|^{1/2} |P| + |V|).
\end{equation}
Inserting these bounds in the calculations that follow retrieves the lower bound of  Stevens and de Zeeuw on pinned distances. This is why we need Proposition~\ref{Pts vars inc}. To prove it we will look at triple intersections of varieties.

In place of Proposition \ref{Pts vars inc}, Lund, Sheffer, and de Zeeuw gave the bound
\begin{equation}\label{eq:LSZincBound}
I(\Pi_r, \vrts_r) = O_\eps(M^{2/5}|\Pi_r|^{7/5 + \eps} + M|\Pi_r|),
\end{equation}
for any $\eps > 0$, with the implied constant depending on $\eps$.
They prove this bound by combining (\ref{eq:KSTbounds}) with the method of polynomial partitioning.
Polynomial partitioning is not available over general fields, but more importantly it relies on the asymmetry between the dependence on $|P|$ and $|V|$ in (\ref{eq:KSTbounds}).
Since the main term in Proposition \ref{Pts vars inc} is symmetric in its dependence on $|P|$ and $|V|$, it cannot be improved by polynomial partitioning, even in the case $\F = \mathbb{R}$.
Note that Proposition \ref{Pts vars inc} is  stronger than (\ref{eq:LSZincBound}) for $M \geq |\Pi_r|^{1/4}$.

\subsection{Isotropic lines}
One difficulty that arises when working over arbitrary fields is that there may be nonzero vectors $v$ with $\|v\| = 0$.
Indeed, suppose that $i = \sqrt{-1}$ is an element of $\F$.
Then, $\|(1,i)\| = \|(1,-i)\| = 0$.
These are {\em isotropic vectors}.

Let $v \in F^2$ be isotropic, and let $a \in F^2$ be a arbitrary.
Let $\ell$  be the line defined by $a + \lambda v$, where $\lambda$ ranges over $\F$.
For any $b,c \in \ell$, the vector $b-c$ is isotropic.
In this case, $\ell$ is an {\em isotropic line}.

In general, if $\F$ includes $\sqrt{-1}$, then there are two isotropic lines through each point.
Otherwise, there are no isotropic lines.
We collect here a few facts about isotropic lines.

\begin{lemma}\label{lem:no011triangle}
Suppose that $a \neq b$ and $\|a-b\| = 0$.
Then, either $\|a\| = \|b\| = 0$, or $\|a\| \neq \|b\|$.
\end{lemma}

\begin{proof}
Suppose, for contradiction, that $\|a\| = \|b\| \neq 0$.
Then, there is a vector $c$ with $\|c\| = 1$ and $c \cdot b = 0$.
Note that $b$ and $c$ are a basis for $\F^2$.
Let $a = \lambda_1 b + \lambda_2 c$.
Then,
$$0 = \|a-b\| = \|a\| + \|b\| - 2 (a \cdot b) = 2\|a\| - 2(a \cdot b),$$
so 
$$\|a\| = a \cdot b = (\lambda_1 b + \lambda_2 c) \cdot b = \lambda_1 \|a\| +  \lambda_2 (c \cdot b) = \lambda_1 \|a\|.$$
Hence, $\lambda_1 = 1$.
Using this,
$$\|a\| = \|b + \lambda_2 c\| = \|b\| + \lambda_2^2 \|c\| + 2\lambda_2 (b \cdot c) = \|a\| + \lambda_2^2,$$
so $\lambda_2 = 0$.
Hence $a=b$, which contradicts assumption of the lemma.
\end{proof}

\begin{lemma}\label{lem:noIsotropicBisector}
Let $a \neq b$, $\|a\| = \|b\| \neq 0$, $c \neq 0$, $\|c\| = 0$. Then, $\|a - c\| \neq \|b - c\|$.
\end{lemma}
\begin{proof}
$$\|a-c\| = \|a\| + \|c\| - 2(a \cdot c) = \|a\| - 2(a \cdot c).$$
Similarly,
$$\|b-c\| = \|a\| - 2(b \cdot c).$$

Suppose, for contradiction, that $\|a-c\| = \|b-c\|$.
Then, $a \cdot c = b \cdot c$.
Hence, $(a-b) \cdot c = 0 = c \cdot c$.
Since $\cdot$ is a nondegenerate bilinear form, this implies that there is some $\lambda \neq 0$ so that $a-b = \lambda c$.
Hence, $\|a-b\| = 0$.
Applying Lemma \ref{lem:no011triangle} to the vectors $a,b$ leads to a contradiction.
\end{proof}

\begin{corollary}\label{cor:noIsotropicBisector}
Perpendicular bisectors are not isotropic.
\end{corollary}
\begin{proof}
Suppose that $\ell$ is the perpendicular bisector of $a,b$.
By the definition of a perpendicular bisector, we exclude the possibility that $\|a - b\| = 0$.
Consequently, if $c \in \ell$, then $\|a-c\| = \|b-c\| \neq 0$.
By translating, we may assume that $0 \in \ell$, so that $\|a\| = \|b\| \neq 0$.
Let $c \in \ell$ be a non-zero point, so that $\|a-c\| = \|b-c\| \neq 0$.
Then, by Lemma \ref{lem:noIsotropicBisector}, since $\|a\|=\|b\| \neq 0$, we have that $\|c\| \neq 0$.
Hence, $\ell$ is not isotropic.
\end{proof}

\subsection{Rotations and reflections}

The proof of Lemma \ref{lem:LSdZ} relies on the properties of transformations of $\F^2$ that preserve $\|a-b\|$ for all $a,b \in \F^2$.
We call these transformations {\em isometries}.

We assume a few basic facts about the orthogonal group associated to an arbitrary quadratic form over a field with characteristic not equal to $2$; see, for example, chapters 4-6 in \cite{grove2002classical}.
Any linear transformation of $\F^2$ that preserves $\|\cdot\|$ belongs to an associated orthogonal group.
As matrices, the orthogonal transformations have determinant $1$ or $-1$.
If the determinant is $1$, the transformation is {\em positively oriented}.

If $\tau$ is an orthogonal transformation, then $a \rightarrow \tau(a) + v$ preserves $\|a-b\|$ for all $a,b \in \F^2$.
Positively oriented orthogonal transformations are rotations around the origin and the identity, so the positively oriented isometries of $\|a-b\|$ are rotations around any point in $\F^2$, pure translations, and the identity.

For any vector $u$ with $\|u\| \neq 0$, the orthogonal transformation $\sigma_u$ with
$$\sigma_u(a) = a - 2 \frac{a \cdot u}{\|u\|} u$$
is the reflection over the non-isotropic line
$$\{v: u \cdot v = 0\}$$
through the origin.
Note that $\sigma_u(u) = -u$ and $\sigma_u(v) = v$ if $u \cdot v = 0$.
Note also that $\sigma_u(a) = \sigma_{\lambda u}(a)$ for $\lambda \neq 0$.
Reflections over lines through the origin have order $2$, and the composition of reflections is positively oriented.

More generally, let $u$ with $\|u\| \neq 0$, let $w$ with $u \cdot w = 0$, and let $\lambda_1 \in \F$.
The transformation
\[r_\ell(a) = \sigma_u(a)+ 2\lambda_1 u\]
is a reflection over the line
\[\ell = \{\lambda_1 u + v : u \cdot v = 0\}.  \]
In general, for any non-isotropic line $\ell$, we denote the reflection over $\ell$ by $r_\ell$.

\begin{lemma} \label{Lemma 9}
	For any two points $a,b \in \F^2$ with $\|a-b\| \neq 0$ and non-isotropic line $\ell$, we have $r_\ell(a) = b$ if and only if $\ell = \bis(a,b)$.
\end{lemma}
\begin{proof}
	If $r_\ell(a) = b$ and $p \in \ell$, then $\|p - a\| = \|p - b\|$ since $r_\ell$ is an isometry.
	
	It remains to show that such an $r_\ell$ always exists.
	Let $u = a-b$.
	It is always possible to choose $\lambda$ so that 
	\[b = a - 2 \left (\frac{a \cdot (a-b)}{\|a-b\|} - \lambda \right ) (a - b),\]
	which shows that $r_\ell(a) = b$ for the line $\ell = \{\lambda (a-b) + v: (a-b) \cdot v = 0\}$.
\end{proof}

\begin{lemma} \label{Lemma 10}
The composition of a pair of reflections is a positively oriented isometry of $\|a-b\|$.
In particular, $r_\ell  \circ r_{\ell'}$ is the identity if $\ell = \ell'$, is a translation if $\ell$ is parallel to $\ell'$, and is otherwise a rotation around $\ell \cap \ell'$.
\end{lemma}
\begin{proof}

Let 
$r_\ell = \sigma_u(a) + 2 \lambda u$, and
$r_{\ell'} = \sigma_{u'}(a) + 2 \lambda' u'$ for some $u,u'$ depending on $\ell, \ell'$.
Then
\[r_\ell \circ r_\ell'(a) = \sigma_u \sigma_{u'}(a) + 2 \lambda' \sigma_u(u') + 2 \lambda u.\]
This is a positively oriented orthogonal transformation $\sigma_u \sigma_{u'}$ composed with a translation that does not depend on $a$.

Given distinct points $a$ and $a'$, there is a unique line $\ell$ such that $r_\ell(a) = a'$.
Consequently, if $r_\ell r_{\ell'}(a) = a$, then either $\ell = \ell'$ or $a$ is a fixed point of both $r_\ell$ and $r_{\ell'}$, and hence must be in their intersection.
The lemma follows because the only positively oriented isometries without fixed points are pure translations and the only positively oriented isometries with a single fixed point are pure rotations.
\end{proof}

\begin{lemma}
Given any quadruple of points $a,c,a',c'$ with $\|a-c\| = \|a'-c'\| \neq 0$, there is a unique positively oriented isometry $\tau$ such that $(\tau(a) , \tau(c) ) = (a',c')$.
\end{lemma}
\begin{proof}
	First translate $a$ to $a'$, then apply the unique rotation that takes $c$ to $c'$.
\end{proof}

We are now ready to prove Lemma \ref{lem:LSdZ}.

\begin{proof}[Proof of Lemma \ref{lem:LSdZ}]

Let $(b,d) \in S_{ac} \cap S_{a'c'}$.
Since $\bis(a,b)$ is defined only when $\|a-b\| \neq 0$, we can assume that either $a=b$ or $\|a-b\| \neq 0$, and likewise for the pairs $(a',b), (c,d),$ and $(c',d)$.

If $a \neq b$, then let $\ell_1 = \bis(a,b)$.
Otherwise, if $c \neq d$, then let $\ell_1 = \bis(c,d)$.
Otherwise, let $\ell_1$ be the line through the pair $(a,c) = (b,d)$.
Define $\ell_2$ similarly, using $(a',b')$ in place of $(a,b)$.
Note that we must have $\ell_1 \neq \ell_2$, since if they were equal we would have $(a,c) = (a',c')$.

Let $\tau$ be the unique rotation or translation that carries $(a,c)$ to $(a',c')$.
Since $(b,d) \in S_{ac} \cap S_{a'c'}$, we have by Lemma~\ref{Lemma 9} that $\tau$ is the composition of a reflection over $\ell_1$ followed by a reflection over $\ell_2$.
If $\tau$ is a translation, then $\ell_1$ must be parallel to $\ell_2$, so $b$ is in the line that contains $(a,a')$ and $d$ is in the line that contains $(c,c')$.
If $\tau$ is a rotation, then, by Lemma~\ref{Lemma 10}, $\ell_1 \cap \ell_2$ must be the fixed point $p$ of $\tau$.
If $p = a$, then $a = b = a'$.
Otherwise, $\|a-p\| = \|b-p\| = \|a'-p\| \neq 0$, and so $b$ is contained in the circle with center $p$ that contains $(a,a')$.
The very same reasoning applies for $d$, using $(c,c')$ in place of $(a,a')$.
\end{proof}

\subsection{Bounding $\en$}

Having established Lemma \ref{lem:LSdZ}, we are ready to prove Proposition \ref{Pts vars inc}.

\begin{proof}[Proof of Proposition~\ref{Pts vars inc}]
We set $|P|=m$ and $|V| = n$. We also denote ordered pairs of points $(b,d) \in P$ by $x$ and ordered pairs of points $(a,c)$ such that $S_{ac} \in V$ by $v$. For each $x \in P$ let
\[
i(x) = |\{ S_{v} \in V : x \in S_{v} \}|
\]
the number of varieties it is incident to. Applying H\"older's inequality we get
\[
I(P,V) = \sum_{x \in P} i(x) \leq m^{2/3} \left(  \sum_{x \in P} i(x)^3 \right)^{1/3}.
\]
We next use the fact that $i^3 = O(1 + i(i-1)(i-2))$ to obtain
\begin{align*}
I(P,V) 
& = O\left( m^{2/3} \left(m + \sum_{x \in P} i(x) (i(x)-1) (i(x)-2)  \right)^{1/3} \right) \\
& =  O\left( m + m^{2/3} \left(\sum_{x \in P} i(x) (i(x)-1) (i(x)-2)  \right)^{1/3} \right).
\end{align*}

We next expand the bracketed term.
\begin{align*}
\sum_{x \in P} i(x) (i(x)-1) (i(x)-2)  
& = \sum_{x \in P}  \sum_{\substack{\text{pairwise distinct} \\ v_1, v_2 , v_3}} 1_{x \in S_{v_1}} 1_{x \in S_{v_3}} 1_{x \in S_{v_3}}  \\
& = \sum_{x \in P}  \sum_{\substack{\text{pairwise distinct} \\ v_1, v_2 , v_3}} 1_{x \in S_{v_1} \cap S_{v_2} \cap S_{v_3}}  \\
& \leq \sum_{\substack{\text{pairwise distinct} \\ v_1, v_2 , v_3}}  \sum_{x \in \F^4} 1_{x \in S_{v_1} \cap S_{v_2} \cap S_{v_3}}  \\
& = \sum_{\substack{\text{pairwise distinct} \\ v_1, v_2 , v_3}}  |S_{v_1} \cap S_{v_2} \cap S_{v_3}|.
\end{align*}

In view of Lemma~\ref{lem:LSdZ} we observe that $(b,d) = x \in S_{v_1} \cap S_{v_2} \cap S_{v_3}$ if and only if $x$ lies in all three intersections $S_{v_i} \cap S_{v_j}$ with $1 \leq i < j \leq 3$. Therefore, by Lemma~\ref{lem:LSdZ}, if $x \in S_{v_1} \cap S_{v_2} \cap S_{v_3}$, then both $b$ and $d$ belong to the intersection of three circles or lines. We may solely focus on $b$ because knowing $b$ and that $x \in S_{v}$ uniquely determines (the common perpendicular bisector and hence) $d$. There are two separate cases to consider.

The first case is when all three lines or circles are the same. Then there are at most $M$ possibilities for $x=(b,d)$ because $b$ is incident to this line or circle (there is only one possibility for $d$). There are at most $2 n^2 M$ possibilities for the $v_i = (a_i,c_i)$ because once $v_1$ and $v_2$ have been selected, they determine the circles or lines on which $a_3$ and $c_3$ must be incident to. Hence there are at most $M$ possibilities for $a_3$. Since $|a_3c_3| = r$, $c_3$ belongs to the intersection of two circles and there are at most two possibilities for $c_3$. In total, the contribution to the bracketed term coming from this case is $O(M^2 n^2)$.

The second case is when at least two of the lines or circles are distinct. Then $b$ lies in the intersection of two lines or circles, so there are at most two possibilities for $b$ and hence for $x$. There are at most $n^3$ possibilities for the $v_i = (a_i,c_i)$. In total, the contribution to the bracketed term coming from this case is $O(n^3)$.

Putting everything together yields
\begin{align*}
I(P,V) 
& =  O\left( m + m^{2/3} \left( \sum_{\substack{\text{pairwise distinct} \\ v_1, v_2 , v_3}}  |S_{v_1} \cap S_{v_2} \cap S_{v_1}|  \right)^{1/3} \right) \\
& =  O\left( m + m^{2/3} \left(M^{2} n^{2} + n^{3}  \right)^{1/3} \right) \\
& =  O\left( m + M^{2/3} m^{2/3} n^{2/3} + m^{2/3} n \right).
\end{align*}
\end{proof}

Summing over the distances in $\F$ gives a bound on $\en$, as follows.

\begin{corollary} \label{Q bound}
Let $A \subseteq \F^2$. If, for each $r \in \F^\times$, $m_r = |\{ (a,b) \in A \times A : \| a - b \| = r\} |$, then
\[
\en = O\left( M^{2/3} |A|^{4/3} \left(\sum_{r \in \F^\times} m_r^2 \right)^{1/3}  + |A|^{1/3} \left(\sum_{r \in \F^\times} m_r^2 \right)^{2/3} \right).
\]
\end{corollary}
\begin{proof}
We saw above that
\[
\en = \sum_{r \in \F^\times} I(\Pi_r, \vrts_r).
\]

Applying  Proposition~\ref{Pts vars inc} with $P = \Pi_r$ and $V = \vrts_r$, and noting that $|\Pi_r| = |\vrts_r| = m_r$, gives
\[
\en = O\left( \sum_{r \in \F^\times} M^{2/3} m_r^{4/3} + m_r^{5/3} \right) = O\left( \sum_{r \in \F^\times} M^{2/3} m_r^{2/3} m_r^{2/3} + m_r^{1/3} m_r^{4/3} \right).
\]
Applying H\"older's inequality twice gives 
\[
\en = O\left( M^{2/3} \left(\sum_{r \in \F^\times} m_r \right)^{2/3} \left(\sum_{r \in \F^\times} m_r^2 \right)^{1/3}  + \left(\sum_{r \in \F^\times} m_r \right)^{1/3} \left(\sum_{r \in \F^\times} m_r^2 \right)^{2/3} \right).
\]
The final observation is $\sum_{r \in \F^\times}  m_r \leq |A|^2$ (because each ordered pair of vertices contributes at most 1 to the sum) and the claim follows.
\end{proof}

\subsection{Distance quadruples and isosceles triangles}
The next step in the proof of Theorem~\ref{Pinned dist ff} is to obtain an upper bound for the sum $\sum_r m_r^2$. This is step is fairly standard. It relies on a point-line incidence bound of Stevens and de Zeeuw. We begin by bounding this sum in terms of
\[
\IT = |\{(a,b,c) \in A \times A \times A : \|a-b\| = \|a-c\| \neq 0 \}|,
\]
which equals the number of ordered isosceles triangles with all three vertices in $A$.
From Lemma \ref{lem:no011triangle}, we know that, if $(a,b,c) \in \IT$, then $\|b-c\| \neq 0$.

\begin{lemma} \label{IT to L2}
	Let $A \subseteq \F^2$ and $\IT$ the number of ordered isosceles triangles with all three vertices in $A$. If $m_r$ is the number of ordered pairs of elements of $A$ a distance $r$ apart, then
	\[
	\sum_{r \in \F^\times} m_r^2 \leq |A| \, \IT.
	\]
\end{lemma}

\begin{proof}
	For each $a \in A$ and $r \in \F^\times$ let $f_{r,a}: A \to \{0,1\}$ be the function defined by $f_{r,a}(b) = 1$ if $\| a - b \|=r$ and $f_{r,a}(b) =0$ otherwise. The following string of inequalities proves the result.
	\[
	\sum_{r \in \F^\times} m_r^2 =  \sum_{r \in \F^\times} \left( \sum_{a \in A} \sum_{b \in A} f_{r,a}(b) \right)^2 \leq |A| \sum_{r \in \F^\times} \sum_{a \in A} \left(  \sum_{b \in A} f_{r,a}(b) \right)^2 = |A| \, \IT.
	\]
\end{proof}

To bound $\IT$ we apply a slight modification of the aforementioned point-line incidence bound of Stevens and de Zeeuw.

\begin{theorem}[Stevens and de Zeeuw] \label{SdZ}
	Let $P \subset \F^2$ be a set and $\lines$ be a collection of lines. Suppose that $|P| \leq p^{8/5}$. The number of incidences between points in $P$ and lines in $\lines$ is
	\[
	I(P, \lines) = O(|P|^{11/15} |\lines|^{11/15} + |P| +|\lines|).
	\]
\end{theorem}
\begin{proof}
	We let $|P|=m$ and $|\lines|=n$. Stevens and de Zeeuw proved $I(p, \lines) = O((mn)^{11/15})$ under the restrictions $m^{-2} n^{13} \leq p^{15}$ and $m^{7/8} \leq n \leq m^{8/7}$. Their result verifies the claimed bound when $m^{7/8} \leq n \leq m^{8/7}$ and $m \leq p^{8/5}$. In all other cases the unconditional bounds
	\[
	I(P, \lines) = O(m^{1/2} n + m ) \text{ or } I(P, \lines) = O(m^{1/2} n + m ) 
	\] 
	imply the theorem.
\end{proof}

\begin{lemma} \label{IT}
	Let $A \subseteq \F^2$, $\en$ be the bisector energy of $A$, and $\IT$ be the number of ordered isosceles triangles with all three vertices in $A$. If $|A| \leq p^{8/5}$, then either $	Q \leq \log^{15/4}(|A|) \, |A|^{5/4}$ or
	\[
	\IT = O(|A|^{5/3} \en^{4/15}).
	\]
\end{lemma}

\begin{proof}
	We denote $|A|$ by $m$. We use the following identity
	\[
	\IT = \sum i(\ell) b(\ell),
	\]
	where $b(\ell)$ is the number of ordered pairs of elements of $A$ whose perpendicular bisector is $\ell$, $i(\ell)$ is the number of points in $A$ incident to $\ell$, and the sum is over all nonisotropic lines $\ell \subset \F^2$ with $b(\ell) \geq 1$. 
	The identity is true because $\|a-b\| = \|a-c\|$ is equivalent to $a$ being incident to the perpendicular bisector of $(b,c)$. We will break up the above sum in two parts, depending on whether $b(\ell) \leq \en / m^2 $, and apply dyadic decomposition to each part. 
	
	For $k \geq 1$, set $\lines_k$ be the collection of nonisotropic lines $\ell \subset \F^2$ that satisfy $k \leq b(\ell) < 2k$. We decompose the sum in dyadic intervals. To simplify the notation we call integers of the form $2^j$ for some $j\geq 0$ dyadic (note the smallest dyadic integer is 1). 
	\[
	\IT = \sum_{\ell  : i(\ell) >1} i(\ell) b(\ell) \leq \sum_{\text{dyadic } k \leq m} 2k \sum_{\ell \in \lines_k} i(\ell) =  \sum_{\text{dyadic } k \leq m} 2k \, I(A, \lines_k).
	\]
	We wish to apply the theorem of Stevens and de Zeeuw and so we need an upper bound on $|\lines_k|$.
	Note that $m^2$ is equal to $\sum b(\ell)$ minus the number of pairs of points on isotropic lines, and in particular $m^2 \geq \sum b(\ell)$.
	Hence,
	\[m^2 \geq \sum b(\ell) \geq k |\lines_k|.\]
	Also note that
	\[\en = \sum b(\ell)^2 \geq k^2 |\lines_k|.\]
	Combining these bounds,
	\[
	|\lines_k| \leq \min\left\{ \frac{m^2}{k} , \frac{\en}{k^2} \right\}.
	\]
	We apply the former upper bound to dyadic integers $1 \leq k \leq \tfrac{\en}{m^2}$ and the latter to $ \tfrac{\en}{m^2} \leq k \leq m$.
	\begin{align*}
	\IT 
	& \leq  \sum_{\text{dyadic } k \leq \en/m^2}  2k \, I(A, \lines_k) + \sum_{\text{dyadic } k \in (\en/m^2 , m]} 2k \, I(A, \lines_k) \\
	& = O\left( \sum_{\text{dyadic } k \leq \en/m^2} k( m^{11/15} |\lines_k |^{11/15} + m + |\lines_k |) + \sum_{\text{dyadic } k \in (\en/m^2, m]} k ( m^{11/15} |\lines_k |^{11/15} + m + |\lines_k|) \right) \\
	& = O\left( \sum_{\text{dyadic } k \leq \en/m^2} (m^{33/15} k^{4/15} + km + m^{2}) + \sum_{\text{dyadic } k \in (\en/m^2 , m]} (m^{11/15} \en^{11/15} k^{-7/15}+ km + \en k^{-1}) \right) \\
	& = O\left( m^2 \, \log(m) + m^{33/15} (\en /m^{2}) ^{4/15} + m^{11/15} \en^{11/15} (\en/m^2)^{-7/15} \right) \\
	& = O(m^2 \, \log(m) + m^{5/3} \en^{4/15}).
	\end{align*}
	If $\en > \log^{15/4}(m) \, m^{5/4}$, then $m^2 \, \log(m) < m^{5/3} \en^{4/15}$ and the lemma follows.
\end{proof}

\subsection{Finishing up}
Applying the trivial, yet sometimes sharp, bound $\en \leq |A|^3$ to Lemmata~\ref{IT to L2}~\&~\ref{IT} leads to the Stevens and de Zeeuw bound on $\sum m_r^2$. We will use Proposition~\ref{Pts vars inc} to characterize sets with nearly maximum bisector energy. This characterization is crucial for the proof of Theorem~\ref{Pinned dist ff}. The qualitative statement is that a set with nearly maximum bisector energy must have many points incident to a single circle or line. We state it in terms of a dichotomy.
\begin{corollary} \label{Dichotomy}
	Let $A \subseteq \F^2$, $\en$ be the bisector energy of $A$, and $M$ be the maximum number of points of $A$ incident to any circle or line. Either $\en = O(|A|^{3-1/37})$ or $M \geq |A|^{27/37}$.
\end{corollary}

\begin{proof}
	Denoting $|A|$ by $m$, and combining Corollary~\ref{Q bound} and Lemmata~\ref{IT to L2}~\&~\ref{IT} (in Lemma~\ref{IT} we may assume that $\en > \log^{14/5}(m) m^{5/4}$) gives
	\[
	\en = O(M^{2/3} m^{4/3} \en^{4/45} m^{8/9} + m^{2/3} \en^{8/45} m^{16/9}) = O(M^{2/3} m^{20/9} \en^{4/45} + m^{22/9} \en^{8/45}).
	\]
	This implies 
	\[
	\en = O(M^{30/41} m^{100/41} + m^{110/37}).
	\]
	If $M \leq m^{27/37}$, then second term dominates, and the theorem follows.
\end{proof}

We are now in position to prove Theorem~\ref{Pinned dist ff}.

\begin{proof}[Proof of Twheorem~\ref{Pinned dist ff}]
	
Recall that $M$ is the maximum number of points of $A$ incident to any circle or line. Using an element of $A$ on a non-isotropic circle or line incident to $M$ points of $A$ gives at least $(M-1)/2$ pinned distances. If the circle or line incident to $M$ points is isotropic, we use as a pin the point of $A$ not on it and obtain at least $M$ pinned distances. In both cases, if the second conclusion of Corollary~\ref{Dichotomy} applies, then we have at least a constant multiple of $N^{27/37} \geq N^{20/37}$ pinned distances and we are done. We may therefore assume $\en = O(N^{110/37})$. 

Cauchy--Schwarz together with an averaging argument yields there exist at least nearly $N^3/ \IT$ pinned distances. Recall that $f_{r,a}(b)$ is the indicator function for $\|a-b\| = r$. We use $\delta_{a}$ to denote the number of nonzero distances $r$ at which there is a point $b \in A$ such that $\|a-b\| = r$. We use $\Delta_{\text{pin}}(A) = \max_a \delta_a$ to denote the pinned distance number. Since there are at most $2$ isotropic lines through each point, for each $a \in A$ there are at most $2M$ points $b$ such that $\|a - b\| = 0$.
\[ \Delta_{\text{pin}}(A) \IT \geq \sum_{a \in A} |\delta_{a}| \sum_{r \in \F^\times} \left (\sum_{b \in A} f_{r,a}(b) \right )^2 \geq \sum_{a \in A}  \sum_{r \in \F^\times} \sum_{b \in A} f_{r,a}(b) \geq N(N - 2M -1)^2 = \Omega(N^3).\]

Lemma~\ref{IT} implies there are $\Omega(N^{4/3} / \en^{4/15})$ pinned distances. The hypothesis $\en = O(N^{110/37})$ implies this number is $\Omega(N^{20/37})$.
\end{proof}

\section[Proof of Theorem 2 and sketch of proof of Theorem 3]{Proof of Theorem \ref{Real bis en} and sketch of proof of Theorem \ref{thm:distinctBisectors}}

We will need the following slight generalization of Theorem \ref{Real bis en}.

\begin{theorem}\label{thm:refinedRealBisEnergy}
	There exists an absolute positive constant $C>0$ with the following property.
	Let $A \subset \RR^2$ be a finite set of size $|A| = N$, and let $R_M \subset A^2$ be the set of pairs of points of $A$ that are not contained in any line or circle that contains more than $M$ points of $A$.
	If
	\[\en_M = |\{ (a,b,c,d) \in A^4 : (a,b) \in R_M, \, \bis(a,b) = \bis(c,d) \}|, \]
	then
	\[\en_M \leq C(MN^2 + \log^{1/2}(N) \, N^{2+1/2}). \]
\end{theorem}

If no circle or line contains more than $M$ points of $A$, then $\en_M = \en$, and we recover Theorem \ref{Real bis en}.
Theorem \ref{thm:distinctBisectors} follows from the arguments in \cite{lund2016refined} by using Theorem \ref{thm:refinedRealBisEnergy} in place of \cite[Lemma 11]{lund2016refined}.

The main step of the proof of Theorem \ref{thm:refinedRealBisEnergy} is analogous to Proposition \ref{Pts vars inc}, and we rely on the definitions of $\Pi_r$ and $S_r$ given at the beginning of Section \ref{sec:pinnedDistances}.
Given any sets $P \subset \Pi_r$ and $V \subset S_r$, we denote by $I_M(P,V)$ the number of incidences between points of $P_r$ and surfaces in $\vrts_r$, not counting any incidence $((b,d),S_{ac})$ such that the pair $(a,b)$ is contained in a circle or line that contains more than $M$ points of $A$.

\begin{proposition} \label{Real pts vars}
	Let $A \subset \RR^2$.
	If $P \subseteq \Pi_r$, and $V \subseteq \vrts_r$, then
	\[
	I_M(P, V) = O(M|P| + |P|^{3/2} + |P|^{1/2} |V|).
	\] 
\end{proposition}

\begin{proof}
	We set $|P|=m$ and $|V| = n$. We also denote ordered pairs of points $(b,d) \in P$ by $x$ and ordered pairs of points $(a,c)$ such that $S_{ac} \in V$ by $v$. We set $V_1 \subset \RR^2$ to be the ``first coordinate projection'' of $V$: $V_1 = \{ a : S_{ac} \in V \text{ for some $c \in A$}\}$.
	
	For each $(b,d) \in P$ we denote by
	\[
	i((b,d)) = |\{ S_{ac} \in V : (b,d) \in S_{ac}, \, (a,b) \in R_M \}|.
	\]
	Hence
	\[
	I_M(P,V) = \sum_{x \in P} i(x).
	\]
	
	We begin by bounding the contribution coming from those $x$ for which $i(x) \leq CM$ for some sufficiently large (and absolute) $C$. The value of $C$ depends on Beck's theorem \cite{beck1983lattice}, and will only be implicitly determined. Set 
	\[
	P_1 = \{ x \in P : i(x) \leq CM\}.
	\]
	We immediately get
	\[
	I_M(P_1, V) = \sum_{x \in P_1} i(x) \leq CM |P_1| \leq C M m.
	\]
	We are left with bounding the contribution to $I(P,V)$ coming from
	\[
	P_2 = P \minus P_1 = \{ x \in P : i(x) > C M\}.
	\]
	
	A first step is to apply Cauchy--Schwarz for
	\[
	I_M(P_2, V) = \sum_{x \in P_2} i(x) \leq |P_2|^{1/2} \left(\sum_{x \in P_2} i(x)^2 \right)^{1/2} \leq m^{1/2} \left(\sum_{x \in P_2} i(x)^2 \right)^{1/2}.
	\]
	The second, and more substantial, step is to use Beck's theorem to bound $i(x)^2$ for $x \in P_2$. 
	
	Let $x = (b,d) \in P^2$, which we temporarily fix. 
	The point $x$ is incident to $i(x)$ varieties $S_{ac} \in V$ such that $(a,b) \in R_M$.
	It is crucial to note that all these $a$ are distinct (in other words, if $(b,d) \in S_{ac}$ and $(b,d) \in S_{ac'}$, then $c=c'$). This is because if $(b,d) \in S_{ac}$ holds and $a,b,d$ are known, then $c$ is uniquely determined - being the fourth vertex of an isosceles trapezoid. 
	We know that $x$ belongs to $ \tbinom{i(x)}{2}$ intersections of the form $S_{ac} \cap S_{a'c'}$ with $a \neq a'$ and $c \neq c'$. 
	By Lemma~\ref{lem:LSdZ}, for each of these intersections there is a line or circle containing $a,a',$ and $b$. By applying an inversion centered at $b$, we transform these lines or circles to lines. 
	We end up with $i(x)$ points (the images of the $a$), no more than $M < i(x)/C$ of which are contained in any single line. 
	By Beck's theorem~\cite{beck1983lattice}, for sufficiently large $C$, at least $c \,i(x)^2$ lines contain at most $C'$ of the $a$, for some absolute constants $c$ and $C'$. 
	After inverting back to the original setting, $i(x)^2$ is at most $c^{-1}$ times the number of ordered pairs $(v , v') \in V \times V$ such that the circle or line $\Gamma_{vv'}$ in $S_v \cap S_{v'}$ that contains both $a$, $a'$ and $b$ is incident to at most $C$ elements of $V_1$ ($V_1$ is the ``first coordinate projection'' of $V$). 
	
	Summing over all $x \in P_2$ we get
	\[
	\sum_{x \in P_2} i(x)^2 \leq c^{-1} T,
	\]
	where $T$ is the number of ordered triples $(x,v,v') \in P \times V \times V$ such that the circle or line $\Gamma_{vv'}$ in $S_v \cap S_{v'}$ that contains both $a$, $a'$ and $b$ is incident to at most $C'$ elements of $V_1$. We partition in two parts: $T = T_1 \cup T_2$, where $T_1$ is the number of ordered triples where $b$ is the only element of $P_1$ in $\Gamma_{vv'}$ and $T_2$ is the number of ordered triples for which there exist at least two elements of $P_1$ in $\Gamma_{vv'}$. We bound $T_1$ and $T_2$ separately.
	
	Each ordered triple $(x,v,v') \in T_1$ is determined by $v$ and $v'$ and so $T_1 \leq |V|^2 = n^2$.
	
	For each ordered triple in $T_2$ there exist at least two elements $b, b' \in \Gamma_{vv'}$, at least two elements $a \in \Gamma_{vv'}$, and at most $C'$ elements $a \in \Gamma_{vv'}$. Hence there exist $(b,d) , (b'd') \in P$ and $S_{ac} , S_{a'c'}$ such that both $(b,d)$ and $(b',d')$ are incident to both $S_{ac}$ and $S_{a'c'}$. By Lemma~\ref{lem:LSdZ}, the points $(b,d)$ and $(b',d')$ uniquely determine $\Gamma_{vv'}$. Once  $\Gamma_{vv'}$ has been determined there are at most $C'^2$ possibilities for $a$ and $a'$ and hence for $v$ and $v'$. Hence $T_2 \leq C'^2 |P|^2 = C'^2 m^2$.
	
	Therefore
	\begin{align*}
	I(P_2, V) 
	& \leq m^{1/2} \left(\sum_{x \in P_2} i(x)^2 \right)^{1/2} \\
	& \leq c^{-1/2} m^{1/2} (T_1 + T_2)^{1/2} \\
	& \leq c^{-1/2} m^{1/2} (n^2 + C'^2 m^2)^{1/2} \\
	& \leq c^{-1/2} m^{1/2} (n + C' m).
	\end{align*}
	
	Putting everything we get the desired bound:
	\[
	I_M(P,V) = I_M(P_1,V) + I_M(P_2,V) \leq C M m + c^{-1/2} m^{1/2} (n + C' m) = O(M m + m^{3/2} + m^{1/2} n).
	\]
\end{proof}

Theorem~\ref{Real bis en} follows from Proposition~\ref{Real pts vars} and the full strength of the celebrated Guth--Katz theorem on distinct distances \cite{guth2015erdos}.

\begin{proof}[Proof of Theorem~\ref{Real bis en}]
	Using the notation established above and setting $m_r = |\Pi_r| = |\vrts_r|$ we get from Proposition~\ref{Real pts vars}
	\[
	\en_M  = \sum_{r \in \RR} I_M(\Pi_r, \vrts_r) = O\left( \sum_{r} M m_r + m_r^{3/2} \right).
	\]
	It is immediate that 
	\[
	\sum_r m_r = |A|^2 = N^2.
	\]
	Guth and Katz proved the optimal upper bound \cite{guth2015erdos}
	\[
	\sum_r m_r^2 = O(\log(N) N^3).
	\] 
	These imply, via Cauchy-Schwarz,
	\[
	\sum_r m_r^{3/2} \leq \left( \sum_r m_r \right)^{1/2} \left( \sum_r m_r^2 \right)^{1/2} = O(N^{5/2} \log^{1/2}(N)).
	\]
	This finally gives the desired bound 
	\[
	\en_M = O(M N^2 + N^{5/2} \log^{1/2}(N)).
	\]
\end{proof}

\phantomsection

\addcontentsline{toc}{section}{References}

Ben Lund, Department of Mathematics, Princeton University, Princeton, NJ 08544, USA.

\hspace{20pt} Email address: \emph{lund.ben@gmail.com}

Giorgis Petridis, Department of Mathematics, University of Georgia, Athens, GA 30602, USA.

\hspace{20pt} Email address: \emph{giorgis@cantab.net}

\end{document}